\documentclass[12pt]{amsart}

\usepackage{amsfonts,amsmath,amsxtra,amsthm,tikz}
\usepackage{mathrsfs}
\usepackage{times}
\usepackage{anysize}

\marginsize{1in}{1in}{1in}{1.5in}

\usepackage{amscd}
\usepackage{amsthm}
\usepackage{comment}
\usepackage{amssymb}

\usetikzlibrary{arrows}

\DeclareMathOperator{\midd}{mid}
\DeclareMathOperator{\Hilb}{Hilb}
\DeclareMathOperator{\ctot}{ctot}

\newtheorem{lemma}{Lemma}[section]

\newtheorem{theorem}{Theorem}[section]
\newtheorem{corollary}{Corollary} 
 
\theoremstyle{definition}
\newtheorem{definition}{Definition}[section] 
\newtheorem{remark}{Remark}

\begin{document}
\author{Mikhail Mazin}
\address{Mathematics Department, Kansas State University.
Cardwell Hall, Manhattan, KS 66506}
\email{mmazin@math.ksu.edu}
\title[A bijective proof of Loehr-Warrington's formulas]{A bijective proof of Loehr-Warrington's formulas for the statistics $\ctot_{\frac{q}{p}}$ and $\midd_{\frac{q}{p}}$.}
\date{}

\begin{abstract}
Loehr and Warrington introduced partitional statistics $\ctot_{\frac{q}{p}}(D)$ and $\midd_{\frac{q}{p}}(D)$ and provided formulas for these statistics in terms of the boundary graph of the Young diagram $D$. In this paper we give a bijective proof of Loehr-Warrington's formulas using the following simple combinatorial observation: given a Young diagram $D$ and two numbers $a$ and $l,$ the number of boxes in $D$ with the arm length $a$ and the leg length $l$ is one less than the number of boxes with the same properties in the complement to $D.$ Here the complement is taken inside the positive quadrant or, equivalently, a very large rectangle. 

{\bf Keywords:} Partition, Young diagram, Bijection, Hilbert scheme.

{\bf AMS Subject Classificiation Numbers:} 05A17.
\end{abstract}


\maketitle

\section{Introduction.}

Let $D$ be a Young diagram and $(p,q)$ be a pair of positive coprime integers such that $p+q>|D|.$ Following \cite{LW} we introduce the following statistic:

\begin{definition}
For a box $c\in D,$ let $a(c)$ and $l(c)$ denote the lengths of the arm and the leg of $c$ (see Figure \ref{armleg inside the diagram}). The statistic $h_{\frac{q}{p}}(D)$ is defined by
$$
h_{\frac{q}{p}}(D):=\left| \left\lbrace c\in D: \frac{l(c)}{a(c)+1}< \frac{q}{p}< \frac{l(c)+1}{a(c)} \right\rbrace \right|.
$$
\end{definition}

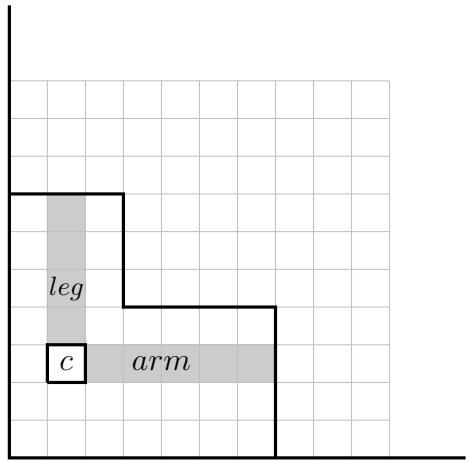
\begin{figure}
\centering
\begin{tikzpicture}
\filldraw [color=gray!50!white, fill=gray!40!white] (1,1)--(1,1.5)--(3.5,1.5)--(3.5,1)--(1,1); 
\filldraw [color=gray!50!white, fill=gray!40!white] (0.5,1.5)--(0.5,3.5)--(1,3.5)--(1,1.5)--(0.5,1.5); 
\draw [step=0.5, color=gray!50!white, very thin] (0,0) grid (5,5);
\draw [very thick] (0,6)--(0,0)--(6,0); 
\draw [very thick] (0,3.5)--(1.5,3.5)--(1.5,2)--(3.5,2)--(3.5,0); 
\draw [very thick] (0.5,1)--(0.5,1.5)--(1,1.5)--(1,1)--(0.5,1); 

\draw (0.75,1.25) node {$c$};

\draw (0.75,2.25) node {\footnotesize $leg$}; 

\draw (2,1.25) node {$arm$}; 

\end{tikzpicture}
\caption{\footnotesize On this example the leg length is $l(c)=4,$ and the arm length is $a(c)=5.$}
\label{armleg inside the diagram}
\end{figure} 

\begin{remark}
The condition $p+q>|D|$ guarantees that $\frac{l(c)}{a(c)+1}\neq \frac{q}{p}$ and $\frac{l(c)+1}{a(c)}\neq\frac{q}{p}$ for all boxes $c\in D.$ In fact, the opposite is also true: for any $n\ge p+q$ there exists a diagram $D$ with area $n$ and a box $c\in D$ such that $\frac{l(c)}{a(c)+1}= \frac{q}{p}.$ There also exists a (possibly different) diagram $D',$ also of area $n,$ and a box $c'\in D'$ such that $\frac{l(c')+1}{a(c')}=\frac{q}{p}.$  
\end{remark}

These statistics play an important role in the theory of Hilbert schemes of points on the complex plane. One can show that the Hilbert scheme of $|D|$ points on the plane can be decomposed into affine cells enumerated by Young diagrams of area $n,$ so that the complex dimension of the cell $C_{D}$ corresponding to the diagram $D$ equals $|D|+h_{\frac{q}{p}}(D).$ (This follows from the Ellingsrud-Str\o mme computation of the character of the torus action on the tangent space to the Hilbert scheme at a monomial ideal \cite{ES}, and the theory of Bia\l ynicki-Birula cell decompositions \cite{BB}.) One gets different cell decompositions for different choices of integers $(p,q),$ but the total number of cells of a given dimension remains the same. One gets the following theorem:

\begin{theorem}[e.g. \cite{LW}]\label{Theorem: equidistributed}
Let $n$ and $h$ be positive integers. Then the number of Young diagrams $D$ of area $n$ and such that $h_{\frac{q}{p}}(D)=h$ is independent of the choice of positive coprime integers $(p,q),$ provided that $p+q>n.$ 
\end{theorem}

For more details on the cell decompositions of Hilbert schemes, see Section \ref{Section: geometric remarks}. Although geometrically Theorem \ref{Theorem: equidistributed} follows immediately from the invariance of the Borel-Moore homology groups of the Hilbert schemes of the plane, combinatorially it is quite puzzling. Loehr and Warrington provided a purely combinatorial proof in \cite{LW}. The strategy of their proof was as follows. Note that for a fixed diagram $D,$\ $h_x(D)$ is a locally constant integer-valued function in $x\in\mathbb R_{>0}\backslash \{\frac{q}{p}: p+q\le |D|\}.$ There are two natural ways to extend this function to all positive reals:

\begin{definition}[\cite{LW}, see Introduction]
The statistic $h_x^+(D)$ (respectively, $h_x^-(D)$) is the continuous on the right (respectively, continuous on the left) extension of $h_x(D).$ In other words, these statistics are defined by formulas:

$$
h_{x}^+(D):=\left| \left\lbrace c\in D:\frac{a(c)}{l(c)+1}\le x<\frac{a(c)+1}{l(c)}\right\rbrace \right|,
$$
and
$$
h_{x}^-(D):=\left| \left\lbrace c\in D:\frac{a(c)}{l(c)+1}<x\le\frac{a(c)+1}{l(c)}\right\rbrace \right|.
$$
\end{definition} 

For each positive rational number $x,$ Loehr and Warrington constructed an explicit bijection from the set of Young diagrams of a given area to itself that interchanges the statistics $h_x^+$ and $h_x^-.$ This provides a combinatorial proof of the fact that all statistics $h_x$ for all $x\in\mathbb R_+$ are equally distributed on diagrams of a given area. Indeed, for a fixed area, there are only finitely many values of $x$ where $h_x$ might jump. By applying Loehr-Warrington's bijections at each such value on the interval $[x,y],$ one gets a bijection interchanging statistics $h_x$ and $h_y.$ Let us recall some definitions from \cite{LW}.

\begin{definition}[\cite{LW}, see Chapter $3$]
Let $D$ be a Young diagram. Let
$$
c^+_{\frac{q}{p}}(D):=\left| \left\lbrace c\in D:\frac{l(c)}{a(c)+1}=\frac{q}{p}\right\rbrace \right|,
$$
$$
c^-_{\frac{q}{p}}(D):=\left| \left\lbrace c\in D:\frac{l(c)+1}{a(c)}=\frac{q}{p}\right\rbrace \right|,
$$
$$
\ctot_{\frac{q}{p}}(D):=c^+_{\frac{q}{p}}(D)+c^-_{\frac{q}{p}}(D),
$$
and
$$
\midd_{\frac{q}{p}}(D):=\left| \left\lbrace c\in D:\frac{l(c)}{a(c)+1}<\frac{q}{p}<\frac{l(c)+1}{a(c)}\right\rbrace \right|.
$$
\end{definition}

\begin{remark}
Note that $h^+_{\frac{q}{p}}(D)=\midd_{\frac{q}{p}}(D)+c^+_{\frac{q}{p}}(D)$ and $h^-_{\frac{q}{p}}(D)=\midd_{\frac{q}{p}}(D)+c^-_{\frac{q}{p}}(D).$
\end{remark}

An important step in the Loehr-Warrington's constructions are the formulas expressing symmetric statistics $\ctot_{\frac{q}{p}}(D)$ and $\midd_{\frac{q}{p}}(D)$ in terms of the boundary graph of the diagram $M(D).$ Let us recall the construction of $M(D).$

Let $K$ be a big enough integer so that $D$ fits into the $Kp\times Kq$ rectangle $R_{Kp,Kq}$ under the diagonal. Let $P=Kp$ and $Q=Kq.$ Consider the boundary lattice path $B(D)$ going from the southeast corner of the rectangle $R_{P,Q}$ to the northwest corner of the rectangle $R_{P,Q}$ along the boundary of the diagram $D.$ We think of $B(D)$ as of an oriented graph with edges labeled by $N$ (northward) and $W$ (westward). Let us label the vertices of $B(D)$ by integers as follows: the starting vertex is labeled by $0$ and then each westward edge adds $q,$ while each northward edge subtracts $p.$ Finally, we identify the vertices labeled by the same integer. The resulting graph is denoted $M(D).$ Note that the graph $M(D)$ comes equipped with an Eulerian tour $E(D),$ following the path $B(D).$ We illustrate this construction on Figure \ref{figure:B(D) and M(D)}. 

\begin{figure}
\centering
\begin{tikzpicture}[scale=1.2]
\filldraw [thick, fill=gray!30!white] (0,0)--(0,3)--(1,3)--(1,2)--(3,2)--(3,1)--(4,1)--(4,0)--(0,0);

\draw [step=0.5, color=gray!50!white, very thin] (0,0) grid (6,4);
\draw [thick] (0,4)--(0,0)--(6,0); 
\draw [thick] (0,4)--(6,0);
\draw [thick] (0,4)--(6,4)--(6,0);

\draw (5.9,-0.1) node {\scriptsize $0$}; 
\draw [thick, ->] (6,0)--(5.5,0);
\draw (5.4,-0.1) node {\scriptsize $2$}; 
\draw [thick, ->] (5.5,0)--(5,0);
\draw (4.85,-0.1) node {\scriptsize $4$}; 
\draw [thick, ->] (5,0)--(4.5,0);
\draw (4.35,-0.1) node {\scriptsize $6$}; 
\draw [thick, ->] (4.5,0)--(4,0);
\draw (3.85,-0.1) node {\scriptsize $8$}; 
\draw [thick, ->] (4,1)--(3.5,1);
\draw (3.35,0.9) node {\scriptsize $4$}; 
\draw [thick, ->] (3.5,1)--(3,1);
\draw (2.85,0.9) node {\scriptsize $6$}; 
\draw [thick, ->] (3,2)--(2.5,2);
\draw (2.4,1.9) node {\scriptsize $2$}; 
\draw [thick, ->] (2.5,2)--(2,2);
\draw (1.85,1.9) node {\scriptsize $4$}; 
\draw [thick, ->] (2,2)--(1.5,2);
\draw (1.35,1.9) node {\scriptsize $6$}; 
\draw [thick, ->] (1.5,2)--(1,2);
\draw (0.85,1.9) node {\scriptsize $8$}; 
\draw [thick, ->] (1,3)--(0.5,3);
\draw (0.35,2.9) node {\scriptsize $4$}; 
\draw [thick, ->] (0.5,3)--(0,3);
\draw (-0.15,2.9) node {\scriptsize $6$}; 

\draw [thick, ->] (4,0)--(4,0.5);
\draw (3.85,0.4) node {\scriptsize $5$}; 
\draw [thick, ->] (4,0.5)--(4,1);
\draw (3.9,0.9) node {\scriptsize $2$}; 
\draw [thick, ->] (0,3)--(0,3.5);
\draw (-0.15,3.4) node {\scriptsize $3$}; 
\draw [thick, ->] (0,3.5)--(0,4);
\draw (-0.1,3.9) node {\scriptsize $0$}; 
\draw [thick, ->] (1,2)--(1,2.5);
\draw (0.85,2.4) node {\scriptsize $5$}; 
\draw [thick, ->] (1,2.5)--(1,3);
\draw (0.9,2.9) node {\scriptsize $2$}; 
\draw [thick, ->] (3,1)--(3,1.5);
\draw (2.85,1.4) node {\scriptsize $3$}; 
\draw [thick, ->] (3,1.5)--(3,2);
\draw (2.85,1.9) node {\scriptsize $0$}; 

\draw (-1,-1.5) node {$0$};
\draw (-1,-1.5) circle [radius=0.2]; 
\draw (1,-1.5) node {$2$}; 
\draw (1,-1.5) circle [radius=0.2]; 
\draw (2,-1.5) node {$3$}; 
\draw (2,-1.5) circle [radius=0.2]; 
\draw (3,-1.5) node {$4$}; 
\draw (3,-1.5) circle [radius=0.2]; 
\draw (4,-1.5) node {$5$}; 
\draw (4,-1.5) circle [radius=0.2]; 
\draw (5,-1.5) node {$6$}; 
\draw (5,-1.5) circle [radius=0.2]; 
\draw (7,-1.5) node {$8$}; 
\draw (7,-1.5) circle [radius=0.2]; 

\draw [arrows={-angle 60}] (-1,-1.3) .. controls (-0.6,-1) and (0.6,-1) .. (1,-1.3);
\draw [arrows={-angle 60}] (-1,-1.3) .. controls (-0.6,-0.7) and (0.6,-0.7) .. (1,-1.3);

\draw [arrows={-angle 60}] (1,-1.3) .. controls (1.4,-1.1) and (2.6,-1.1) .. (3,-1.3);
\draw [arrows={-angle 60}] (1,-1.3) .. controls (1.4,-0.9) and (2.6,-0.9) .. (3,-1.3);
\draw [arrows={-angle 60}] (1,-1.3) .. controls (1.4,-0.7) and (2.6,-0.7) .. (3,-1.3);
\draw [arrows={-angle 60}] (1,-1.3) .. controls (1.4,-0.5) and (2.6,-0.5) .. (3,-1.3);

\draw [arrows={-angle 60}] (3,-1.3) .. controls (3.4,-1.1) and (4.6,-1.1) .. (5,-1.3);
\draw [arrows={-angle 60}] (3,-1.3) .. controls (3.4,-0.9) and (4.6,-0.9) .. (5,-1.3);
\draw [arrows={-angle 60}] (3,-1.3) .. controls (3.4,-0.7) and (4.6,-0.7) .. (5,-1.3);
\draw [arrows={-angle 60}] (3,-1.3) .. controls (3.4,-0.5) and (4.6,-0.5) .. (5,-1.3);

\draw [arrows={-angle 60}] (5,-1.3) .. controls (5.4,-1) and (6.6,-1) .. (7,-1.3);
\draw [arrows={-angle 60}] (5,-1.3) .. controls (5.4,-0.7) and (6.6,-0.7) .. (7,-1.3);

\draw [arrows={-angle 60}] (7,-1.7) .. controls (6.3,-2) and (4.8,-2) .. (4,-1.7);
\draw [arrows={-angle 60}] (7,-1.7) .. controls (6.3,-2.2) and (4.8,-2.2) .. (4,-1.7);

\draw [arrows={-angle 60}] (4,-1.7) .. controls (3.3,-2) and (1.8,-2) .. (1,-1.7);
\draw [arrows={-angle 60}] (4,-1.7) .. controls (3.3,-2.2) and (1.8,-2.2) .. (1,-1.7);

\draw [arrows={-angle 60}] (5,-1.7) .. controls (4.3,-2) and (2.8,-2) .. (2,-1.7);
\draw [arrows={-angle 60}] (5,-1.7) .. controls (4.3,-2.2) and (2.8,-2.2) .. (2,-1.7);

\draw [arrows={-angle 60}] (2,-1.7) .. controls (1.3,-2) and (-0.2,-2) .. (-1,-1.7);
\draw [arrows={-angle 60}] (2,-1.7) .. controls (1.3,-2.2) and (-0.2,-2.2) .. (-1,-1.7);

\end{tikzpicture}
\caption{\footnotesize Example of a boundary path $B(D)$ and the corresponding graph $M(D).$ Here $p=3,$\ $q=2,$ and $K=4.$}
\label{figure:B(D) and M(D)}
\end{figure}
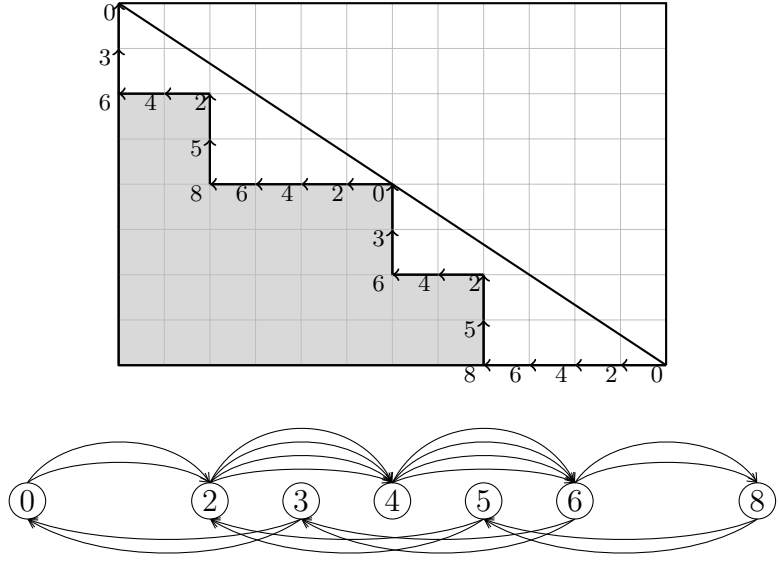

Let $V_M$ be the set of vertices of $M=M(D).$ We identify the vertices of $M$ with the corresponding integers, so that $V_M\subset\mathbb Z.$ For a vertex $v\in V_M$ let $W_{in}(v)$ be the set of westward edges entering $v.$ Respectively, let $N_{in}(v)$ be the set of northward edges entering $v.$ Loehr and Warrington proved the following formulas: 

\begin{theorem}[\cite{LW}]\label{Theorem: LW formulas}
The following formulas for $\ctot_{\frac{q}{p}}(D)$ and $\midd_{\frac{q}{p}}(D)$ in terms of the graph $M(D)$ hold:
\begin{equation}\label{ctot formula}
\ctot_{\frac{q}{p}}(D)=\sum\limits_{v\in V_M} |W_{in}(v)||N_{in}(v)| - K+|N_{in}(0)|
\end{equation}
and
\begin{equation}\label{mid formula}
\midd_{\frac{q}{p}}(D)=|R^+_{P,Q}|-\sum\limits_{v,w\in V_M, v\le w} |W_{in}(v)||N_{in}(w)|,
\end{equation}
where $R^+_{P,Q}\subset R_{P,Q}$ is the set of all boxes below diagonal in $R_{P,Q}.$ 
\end{theorem}

The objective of this paper is to give a simple combinatorial proof of these formulas, based on the following observation. Let us think of $D$ as of a subset in $(\mathbb Z_{\ge 0})^2,$ with the southwest corner box being $(0,0).$ Let $\overline{D}:=(\mathbb Z_{\ge 0})^2\backslash D$ be the complement to $D$ in $(\mathbb Z_{\ge 0})^2$. The arm and the leg lengths for boxes in $\overline{D}$ are defined in same way as for the boxes inside the diagram $D$ (see Figure \ref{armleg outside the diagram}).

\begin{theorem}\label{Theorem: hooks}
Let $D$ be a Young diagram. Let $a$ and $l$ be non-negative integers. Then the number of boxes $c\in D$ inside the diagram such that $a(c)=a$ and $l(c)=c$ is one less than the number of boxes in $\overline{D}$ with the same property. 
\end{theorem}

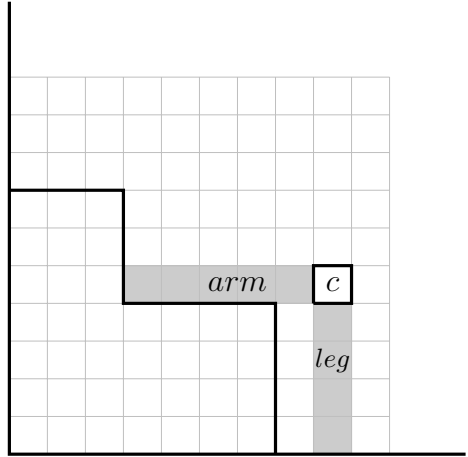
\begin{figure}
\centering
\begin{tikzpicture}
\filldraw [color=gray!50!white, fill=gray!40!white] (1.5,2)--(1.5,2.5)--(4,2.5)--(4,2)--(1.5,2); 
\filldraw [color=gray!50!white, fill=gray!40!white] (4,0)--(4,2)--(4.5,2)--(4.5,0)--(4,0); 
\draw [step=0.5, color=gray!50!white, very thin] (0,0) grid (5,5);
\draw [very thick] (0,6)--(0,0)--(6,0); 
\draw [very thick] (0,3.5)--(1.5,3.5)--(1.5,2)--(3.5,2)--(3.5,0); 
\draw [very thick] (4,2)--(4,2.5)--(4.5,2.5)--(4.5,2)--(4,2); 

\draw (4.25,2.25) node {$c$};

\draw (4.25,1.25) node {\footnotesize $leg$}; 
\draw (3,2.25) node {$arm$};

\end{tikzpicture}
\caption{\footnotesize On this example the leg length is $l(c)=4,$ and the arm length is $a(c)=5.$}
\label{armleg outside the diagram}
\end{figure}

The rest of the paper is organized as follows. Theorem \ref{Theorem: hooks} is proved in Section \ref{Section: inside the positive quadrant}. In Section \ref{Section: inside a rectangle} we adjust the results of Section \ref{Section: inside the positive quadrant} to the case when the diagram is inscribed in a right triangle. In Section \ref{Section: LW identities} we apply the results of Section \ref{Section: inside a rectangle} to give a short proof of Loehr-Warrington's formulas. Finally, in Section \ref{Section: geometric remarks} we discuss a geometric interpretation of the constructions described in this paper, relating these constructions to the geometry of Hilbert schemes.  

\section{Proof of Theorem \ref{Theorem: hooks}.}\label{Section: inside the positive quadrant}

Let $D$ be a Young diagram. As before, let $\overline{D}:=(\mathbb Z_{\ge 0})^2\backslash D$ be the complement to $D$ in the non-negative quadrant, and let $\widehat{D}:=\mathbb Z^2\backslash \overline{D}$ be the complement to $\overline{D}$ in the whole plane $\mathbb Z^2$ (i.e. $\widehat{D}=D\sqcup\{(c,d)\in \mathbb Z^2| c<0\  \mbox{or}\  d<0\}$).  

Consider the set of arrows $A:=\{(a,b)\rightarrow (c,d)|(a,b)\in\overline{D}\ \mbox{and}\ (c,d)\in \widehat{D}\}$ pointing from a box in $\overline{D}$ to a box in $\widehat{D}.$ If two arrows in $A$ differ by a translation by $1$ in vertical or horizontal directions, we say that they are equivalent. This generates an equivalence relation on $A.$ We say that an arrow is {\it escaping} if it is equivalent to an arrow pointing outside the positive quadrant.

Note that there are no north, northeast, or east pointing arrows, and all southwest pointing arrows are escaping. From now on we will concentrate on the set of northwest and west pointing arrows $A_{nw}:=\{[(a,b)\rightarrow (c,d)]\in A| c < a\ \mbox{and}\ d\ge b\}.$ The southeast pointing arrows can be treated similarly.

The following observation can be found in \cite{Hn98}:

\begin{theorem}[\cite{Hn98}]\label{Theorem about boxes inside}
The equivalence classes of non-escaping northwest pointing arrows are in natural one-to-one correspondence with the boxes of the diagram $D.$
\end{theorem}

\begin{proof}
Let us move a northwest pointing arrow to the north and to the west as much as possible. If it is not escaping, there will be a unique representative in the class such that it is impossible to further move it north or west. Indeed, two arrows with the same displacement vector ${\bf v}$ (i.e. the same direction and length) belong to the same equivalence class if and only if their heads can be connected by a lattice path staying inside the intersection $\overline{D}\cap (\widehat{D}+{\bf v}).$ Note that this intersection satisfy the following property: if it contains two boxes in the same row or column, then it contains all the boxes between them. It follows that it is enough to consider paths that do not contain steps in opposite directions. Therefore, two different arrows which cannot be moved north or west cannot be equivalent.  

Suppose that the resulting arrow is $(r,s)\rightarrow (k,m).$ Since it is not escaping, we automatically get $r,s,k,m \ge 0.$ Since we cannot move it north anymore, we have $(k,m+1)\in \overline{D}.$ Since we cannot move it west, we have $(r-1,s)\in \widehat{D}.$ It is not hard to see that there is exactly one such arrow corresponding to each box $(k,s)\in D.$ Indeed, we have $r=k+a(k,s)+1$ and $m=s+l(k,s).$ We illustrate this in Figure \ref{box inside the diagram}.

\begin{figure}
\centering
\begin{tikzpicture}
\filldraw [color=gray!50!white, fill=gray!40!white] (-0.5,1)--(-0.5,1.5)--(4,1.5)--(4,1)--(-0.5,1); 
\filldraw [color=gray!50!white, fill=gray!40!white] (0.5,-0.5)--(0.5,3.5)--(1,3.5)--(1,-0.5)--(0.5,-0.5); 
\filldraw [color=gray!50!white, fill=gray!40!white] (-0.5,3)--(-0.5,3.5)--(1,3.5)--(1,3)--(-0.5,3); 
\filldraw [color=gray!50!white, fill=gray!40!white] (3.5,-0.5)--(3.5,1.5)--(4,1.5)--(4,-0.5)--(3.5,-0.5); 
\draw [step=0.5, color=gray!50!white, very thin] (0,0) grid (5,5);
\draw [very thick] (0,6)--(0,0)--(6,0); 
\draw [very thick] (0,3.5)--(1.5,3.5)--(1.5,2)--(3.5,2)--(3.5,0); 
\filldraw [very thick, fill=gray!30!white] (0.5,1)--(0.5,1.5)--(1,1.5)--(1,1)--(0.5,1);

\draw [thick, <->] (0.75,1.25)--(0.75,3.25);
\draw (0.25,2.25) node {\footnotesize $leg$}; 
\draw [thick, <->] (0.75,1.25)--(3.25,1.25);
\draw (2,0.75) node {$arm$}; 
\draw [very thick, ->] (3.75,1.25)--(0.75,3.25);
\draw [dashed] (3.5,1)--(4,1)--(4,1.5)--(3.5,1.5)--(3.5,1);
\draw [dashed] (0.5,3)--(1,3)--(1,3.5)--(0.5,3.5)--(0.5,3);

\draw (-0.25,1.25) node {$s$};
\draw (0.75,-0.25) node {$k$};
\draw (-0.25,3.25) node {$m$};
\draw (3.75,-0.25) node {$r$};

\end{tikzpicture}
\caption{\footnotesize A northwest pointing arrow which cannot be moved north or west corresponds to a box inside the diagram $D.$}
\label{box inside the diagram}
\end{figure}
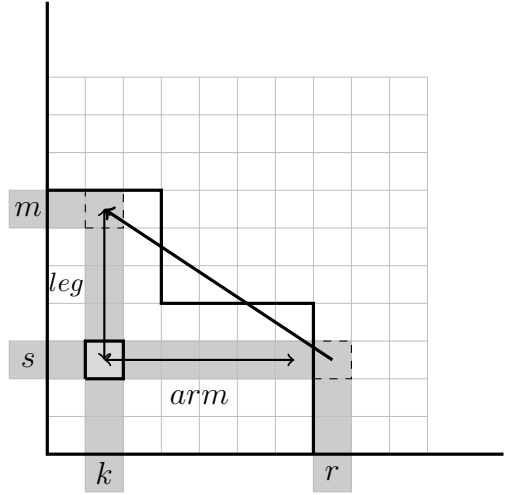 
\end{proof}

One can modify the above construction and get the following result:

\begin{theorem}\label{Theorem about boxes outside}
The equivalence classes of all northwest pointing arrows are in natural one-to-one correspondence with the boxes of the complement $\overline{D}.$
\end{theorem}

\begin{proof}
Let us now move the arrow to the south and to the east as much as possible. Suppose that the resulting arrow is $(r,s)\rightarrow (k,m).$ Since we cannot move it east anymore, we have $(k+1,m)\in \overline{D}.$ Since we cannot move it south, we have $(r,s-1)\in \widehat{D}.$ It is not hard to see that there is exactly one such arrow corresponding to each box $(r,m)\in \overline{D}.$ Indeed, we have $s=m-l(r,m)$ and $k=r-a(r,m)-1.$ We illustrate this on the Figure \ref{box outside the diagram}.
\begin{figure}
\centering
\begin{tikzpicture}
\filldraw [color=gray!50!white, fill=gray!40!white] (-0.5,0)--(-0.5,0.5)--(4.5,0.5)--(4.5,0)--(-0.5,0); 
\filldraw [color=gray!50!white, fill=gray!40!white] (1,-0.5)--(1,2.5)--(1.5,2.5)--(1.5,-0.5)--(1,-0.5); 
\filldraw [color=gray!50!white, fill=gray!40!white] (-0.5,2)--(-0.5,2.5)--(4.5,2.5)--(4.5,2)--(-0.5,2); 
\filldraw [color=gray!50!white, fill=gray!40!white] (4,-0.5)--(4,2.5)--(4.5,2.5)--(4.5,-0.5)--(4,-0.5); 
\draw [step=0.5, color=gray!50!white, very thin] (0,0) grid (5,5);
\draw [very thick] (0,6)--(0,0)--(6,0); 
\draw [very thick] (0,3.5)--(1.5,3.5)--(1.5,2)--(3.5,2)--(3.5,0); 
\filldraw [very thick, fill=gray!30!white] (4,2)--(4,2.5)--(4.5,2.5)--(4.5,2)--(4,2);

\draw [thick, <->] (4.25,0.25)--(4.25,2.25);
\draw (4.75,1.25) node {\footnotesize $leg$}; 
\draw [thick, <->] (1.75,2.25)--(4.25,2.25);
\draw (3,2.75) node {$arm$}; 
\draw [very thick, ->] (4.25,0.25)--(1.25,2.25);
\draw [dashed] (4,0)--(4.5,0)--(4.5,0.5)--(4,0.5)--(4,0);
\draw [dashed] (1,2)--(1.5,2)--(1.5,2.5)--(1,2.5)--(1,2);

\draw (-0.25,0.25) node {$s$};
\draw (1.25,-0.25) node {$k$};
\draw (-0.25,2.25) node {$m$};
\draw (4.25,-0.25) node {$r$};

\end{tikzpicture}
\caption{\footnotesize A northwest pointing arrow which cannot be moved south or east corresponds to a box outside the diagram $D.$}
\label{box outside the diagram}
\end{figure}
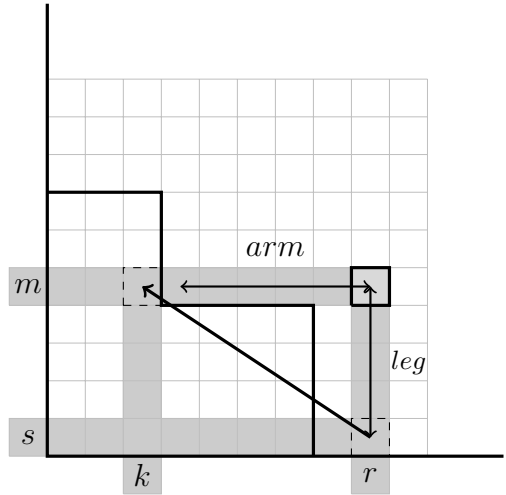
\end{proof}

Note that there is exactly one class of escaping arrows for each fixed direction and length. Note also that direction and length of arrows in an equivalence class are prescribed by the length of the arm and the leg of the corresponding box. More concretely, for a box $c$ the corresponding vector is $(-a(c)-1,l(c)).$ This is valid both for the correspondence from Theorem \ref{Theorem about boxes inside} and the correspondence from Theorem \ref{Theorem about boxes outside}. This completes the proof of the Theorem \ref{Theorem: hooks}.

\section{Inside a Rectangle.}\label{Section: inside a rectangle}

In order to apply the construction from the previous section to prove the Loehr-Warrington's formulas, we need to modify it to deal with the case when the diagram $D$ is inscribed in a right triangle. Let us recall some notations from the Introduction.

Let $(p,q)$ be positive coprime integers. Let $K$ be a big enough integer, so that $D$ fits into the $P\times Q$ rectangle $R_{P,Q}$ under the diagonal, where $P=Kp$ and $Q=Kq.$ In other words, for all boxes $(x,y)\in D$ one has $qx+py\le Kpq-p-q$ (remember that the southwest corner of $D$ is $(0,0)$). As before, let $R_{P,Q}^+:=\{(x,y)\in (\mathbb Z_{\ge 0})^2| qx+py\le Kpq-p-q\}$ denote the set of boxes below the diagonal in $R_{P,Q}$. We get $D\subset R_{P,Q}^+\subset R_{P,Q}.$

When modifying the results of the previous section to this new setup, one runs into an immediate problem: it might happen that the box $c\in \overline{D}$ corresponding to a class of arrows is outside the rectangle $R_{P,Q}.$ This might happen in two cases. First, the arrow might be not steep enough, so that as we move it east its tail moves outside the rectangle. And second, it might be impossible to move an escaping arrow south enough for its head to be below the line $y=Q.$ Fortunately, both problems can be handled if one restricts ones attention to ``steep enough'' arrows only.

\begin{theorem}\label{theorem inside rectangle}
Fix non-negative integers $a$ and $l$ such that $\frac{l}{a+1}\ge \frac{q}{p}.$ Then one has two cases:
\begin{enumerate}
\item If $(a,Q-1-l)\in D,$ then the number of boxes $c$ inside the diagram $D$ such that $l(c)=l$ and $a(c)=a$ is equal to the number of boxes in the complement $R_{P,Q}\backslash D$ with the same property. 
\item If $(a,Q-1-l)\in R_{P,Q}\backslash D,$ then the number of boxes $c$ inside the diagram $D$ such that $l(c)=l$ and $a(c)=a$ is one less than the number of boxes in the complement $R_{P,Q}\backslash D$ with the same property.
\end{enumerate}
\end{theorem}

\begin{proof}
With the condition $\frac{l}{a+1}\ge \frac{q}{p}$ on the slope of arrows, one cannot move a non-escaping arrow so that its tail is outside the triangle $R_{P,Q}^+.$ Indeed, otherwise its head is also above the diagonal, which contradicts the condition $D\subset R_{P,Q}^+.$ Therefore, the only class of arrows that might not be represented by a box in the complement $R_{P,Q}\backslash D$ is the escaping class. 

Now, if $(a,Q-1-l)\in R_{P,Q}\backslash D$ then the arrow $(a,Q-1-l)\rightarrow (-1,Q-1)$ belongs to the escaping class. It is easy to see that in this case the box $c\in\overline{D}$ representing the escaping class is inside the rectangle $R_{P,Q}.$ Otherwise, the box is outside the rectangle. Indeed, if the box is inside then one should be able to move the arrow so that its head is at $(-1,Q-1).$ We illustrate the proof on the Figure \ref{box inside rectangle}.
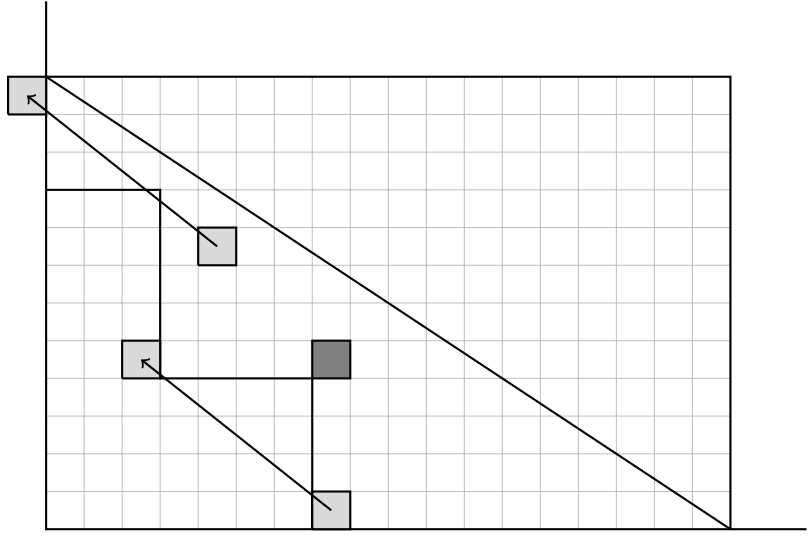
\begin{figure}
\centering
\begin{tikzpicture}

\draw [step=0.5, color=gray!50!white, very thin] (0,0) grid (9,6);
\draw [thick] (0,7)--(0,0)--(10,0); 
\draw [thick] (0,4.5)--(1.5,4.5)--(1.5,2)--(3.5,2)--(3.5,0);
\draw [thick] (0,6)--(9,0);
\draw [thick] (0,6)--(9,6)--(9,0);

\filldraw [thick, fill=gray!30!white] (-0.5,5.5)--(-0.5,6)--(0,6)--(0,5.5)--(-0.5,5.5); 
\filldraw [thick, fill=gray!30!white] (2,3.5)--(2,4)--(2.5,4)--(2.5,3.5)--(2,3.5);
\draw [thick, ->] (2.25,3.75)--(-0.25,5.75);

\filldraw [thick, fill=gray!30!white] (1,2)--(1,2.5)--(1.5,2.5)--(1.5,2)--(1,2); 
\filldraw [thick, fill=gray!30!white] (3.5,0)--(3.5,0.5)--(4,0.5)--(4,0)--(3.5,0);
\draw [thick, ->] (3.75,0.25)--(1.25,2.25);

\filldraw [thick, fill=gray] (3.5,2)--(3.5,2.5)--(4,2.5)--(4,2)--(3.5,2);

\end{tikzpicture}
\caption{\footnotesize Two arrows representing the same escaping class. The top one is $(a,Q-1-l)\rightarrow (-1,Q-1),$ and the bottom one cannot be moved south or east. The class corresponds to the dark gray box in the complement $R_{P,Q}\backslash D$.}
\label{box inside rectangle}
\end{figure}
\end{proof}

Applying Theorem \ref{theorem inside rectangle} to all pairs of numbers $a$ and $l$ satisfying the condition $\frac{l}{a+1}\ge \frac{q}{p}$ one gets the following corollary:

\begin{corollary}\label{GM bijection}
The number of boxes $c$ inside $D$ such that $\frac{l(c)}{a(c)+1}\ge \frac{q}{p}$ plus the number of boxes in $R_{P,Q}^+\backslash D$ is equal to the number of boxes $c'$ in $R_{P,Q}\backslash D$ such that $\frac{l(c')}{a(c')+1}\ge \frac{q}{p}.$
\end{corollary}

\begin{proof}
Indeed, for $(a,Q-1-l)\in R_{P,Q}$ one has 
$$
\frac{l}{a+1}\ge \frac{q}{p}\ \Leftrightarrow\ (a,Q-1-l)\in R^+_{P,Q}.
$$
Therefore, the number of boxes in $R^+_{P,Q}\backslash D$ is equal to the number of pairs $(a,l)$ satisfying the second part of Theorem \ref{theorem inside rectangle}. 
\end{proof}

In our joint paper with Eugene Gorsky \cite{GM11}, we proved this corollary by constructing an explicit bijection in the case when $K=1.$ 

Note that using the southeast pointing arrows instead of northwest, one obtains a similar result about boxes $c\in D$ satisfying $\frac{a(c)}{l(c)+1}\ge \frac{p}{q}$ or, equivalently, $\frac{l(c)+1}{a(c)}\le \frac{q}{p}:$ 

\begin{theorem}\label{theorem inside rectangle southeast}
Fix non-negative integers $a$ and $l$ such that $\frac{l+1}{a}\le \frac{q}{p}.$ Then one has two cases:
\begin{enumerate}
\item If $(P-1-a,l)\in D,$ then the number of boxes $c$ inside the diagram $D$ such that $l(c)=l$ and $a(c)=a$ is equal to the number of boxes in the complement $R_{P,Q}\backslash D$ with the same property. 
\item If $(P-1-a,l)\in R_{P,Q}\backslash D,$ then the number of boxes $c$ inside the diagram $D$ such that $l(c)=l$ and $a(c)=a$ is one less than the number of boxes in the complement $R_{P,Q}\backslash D$ with the same property.
\end{enumerate}
\end{theorem}

\begin{proof}
The same as for Theorem \ref{theorem inside rectangle} with the southeast pointing arrows instead of the northwest.
\end{proof}

Similar to before, one can apply Theorem \ref{theorem inside rectangle southeast} to all pairs of numbers $(a,l)$ satisfying the condition $\frac{l+1}{a}\le \frac{q}{p}$ and get the following corollary:

\begin{corollary}\label{GM bijection southeast}
The number of boxes $c$ inside $D$ such that $\frac{l(c)+1}{a(c)}\le \frac{q}{p}$ plus the number of boxes in $R_{P,Q}^+\backslash D$ is equal to the number of boxes $c'$ in $R_{P,Q}\backslash D$ such that $\frac{l(c')+1}{a(c')}\le \frac{q}{p}.$
\end{corollary}

\section{Loehr-Warrington's identities.}\label{Section: LW identities}

In this section we apply the results of the previous two sections to prove Theorem \ref{Theorem: LW formulas}. Let $D\subset R^+_{P,Q}$ be a Young diagram. We will use the same notations as in the introduction: $B(D)$ is the boundary path, $M(D)$ is the boundary graph, $E(D)$ is the Eulerian tour on $M(D)$ defined by $B(D),$\ $V_M\subset\mathbb Z$ is the set of vertices of $M=M(D).$ For each vertex $v\in V_M,$\  $W_{in}(v)$ is the set of westward edges entering $v.$ Respectively, $N_{in}(v)$ is the set of northward edges entering $v.$ Note that the boxes of the rectangle $R_{P,Q}$ are in natural one-to-one correspondence with pairs of edges of $M(D),$ one northward, and one westward. Indeed, every row of $R_{P,Q}$  contains exactly one northward edge of $B(D),$ and every column contains exactly one westward edge. Moreover, boxes inside $D$ correspond to the pairs for which the northward edge goes before the westward in the Eulerian tour $E(D)$, and boxes in $R_{P,Q}\backslash D$ correspond to the pairs for which the westward edge goes first. The following Lemma follows immediately from the definitions:

\begin{lemma}\label{lemma:slope->edges}
Let $c\in D.$ Suppose that $e_c\in W_{in}(v)$ is the westward edge corresponding to $c,$ and $n_c\in N_{in}(w)$ is the northward edge corresponding to $c.$ Then $v=w+(a(c)+1)q-l(c)p.$ In particular,
\begin{enumerate}
\item One has $v=w$ if and only if $\frac{l(c)}{a(c)+1}=\frac{q}{p},$
\item One has $v<w$ if and only if $\frac{l(c)}{a(c)+1}>\frac{q}{p}.$
\end{enumerate}
Similarly, if $c\in R_{P,Q}\backslash D,$\ $e_c\in W_{in}(v),$ and $n_c\in N_{in}(w),$ then $w=v+a(c)q-(l(c)+1)p.$ In particular,
\begin{enumerate}
\item One has $v=w$ if and only if $\frac{l(c)+1}{a(c)}=\frac{q}{p},$
\item One has $v<w$ if and only if $\frac{l(c)+1}{a(c)}<\frac{q}{p}.$
\end{enumerate}
\end{lemma}

\begin{proof}
The proof is immediate from the definitions. We illustrate it on Figure \ref{figure:slope->edges}.
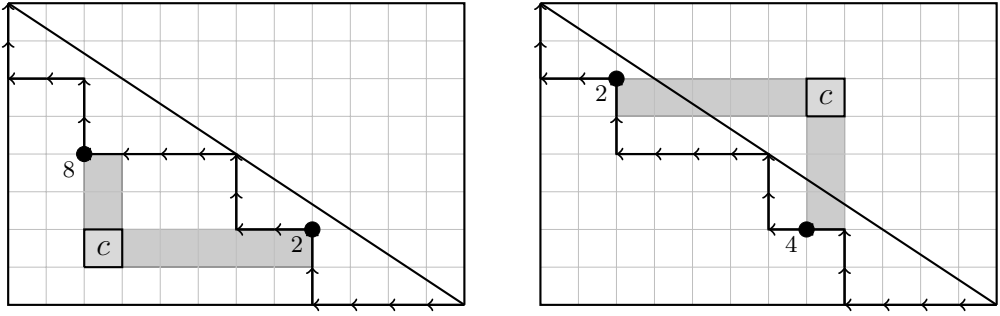
\begin{figure}
\centering
\begin{tikzpicture}
\filldraw [fill=gray!40!white] (1,1)--(1.5,1)--(1.5,2)--(1,2)--(1,1); 
\filldraw [fill=gray!40!white] (1.5,0.5)--(1.5,1)--(4,1)--(4,0.5)--(1.5,0.5); 

\draw [step=0.5, color=gray!50!white, very thin] (0,0) grid (6,4);
\draw [thick] (0,4)--(0,0)--(6,0); 
\draw [thick] (0,3)--(1,3)--(1,2)--(3,2)--(3,1)--(4,1)--(4,0);
\draw [thick] (0,4)--(6,0);
\draw [thick] (0,4)--(6,4)--(6,0);

\draw [thick, ->] (6,0)--(5.5,0);
\draw [thick, ->] (5.5,0)--(5,0);
\draw [thick, ->] (5,0)--(4.5,0);
\draw [thick, ->] (4.5,0)--(4,0);
\draw [thick, ->] (4,1)--(3.5,1);
\draw [thick, ->] (3.5,1)--(3,1);
\draw [thick, ->] (3,2)--(2.5,2);
\draw [thick, ->] (2.5,2)--(2,2);
\draw [thick, ->] (2,2)--(1.5,2);
\draw [thick, ->] (1.5,2)--(1,2);
\draw [thick, ->] (1,3)--(0.5,3);
\draw [thick, ->] (0.5,3)--(0,3);

\draw [thick, ->] (0,3)--(0,3.5);
\draw [thick, ->] (0,3.5)--(0,4);
\draw [thick, ->] (1,2)--(1,2.5);
\draw [thick, ->] (1,2.5)--(1,3);
\draw [thick, ->] (3,1)--(3,1.5);
\draw [thick, ->] (3,1.5)--(3,2);
\draw [thick, ->] (4,0)--(4,0.5);
\draw [thick, ->] (4,0.5)--(4,1);

\filldraw [thick, fill=gray!30!white] (1,0.5)--(1,1)--(1.5,1)--(1.5,0.5)--(1,0.5); 
\draw (1.25,0.75) node {$c$};

\filldraw [fill=black] (1,2) circle [radius=0.1];
\draw (0.8,1.8) node {\scriptsize $8$}; 

\filldraw [fill=black] (4,1) circle [radius=0.1];
\draw (3.8,0.8) node {\scriptsize $2$}; 

\filldraw [fill=gray!40!white] (10.5,1)--(11,1)--(11,2.5)--(10.5,2.5)--(10.5,1); 
\filldraw [fill=gray!40!white] (8,2.5)--(8,3)--(10.5,3)--(10.5,2.5)--(8,2.5); 

\draw [step=0.5, color=gray!50!white, very thin] (7,0) grid (13,4);
\draw [thick] (7,4)--(7,0)--(13,0); 
\draw [thick] (7,3)--(8,3)--(8,2)--(10,2)--(10,1)--(11,1)--(11,0);
\draw [thick] (7,4)--(13,0);
\draw [thick] (7,4)--(13,4)--(13,0);

\draw [thick, ->] (13,0)--(12.5,0);
\draw [thick, ->] (12.5,0)--(12,0);
\draw [thick, ->] (12,0)--(11.5,0);
\draw [thick, ->] (11.5,0)--(11,0);
\draw [thick, ->] (11,1)--(10.5,1);
\draw [thick, ->] (10.5,1)--(10,1);
\draw [thick, ->] (10,2)--(9.5,2);
\draw [thick, ->] (9.5,2)--(9,2);
\draw [thick, ->] (9,2)--(8.5,2);
\draw [thick, ->] (8.5,2)--(8,2);
\draw [thick, ->] (8,3)--(7.5,3);
\draw [thick, ->] (7.5,3)--(7,3);

\draw [thick, ->] (7,3)--(7,3.5);
\draw [thick, ->] (7,3.5)--(7,4);
\draw [thick, ->] (8,2)--(8,2.5);
\draw [thick, ->] (8,2.5)--(8,3);
\draw [thick, ->] (10,1)--(10,1.5);
\draw [thick, ->] (10,1.5)--(10,2);
\draw [thick, ->] (11,0)--(11,0.5);
\draw [thick, ->] (11,0.5)--(11,1);

\filldraw [thick, fill=gray!30!white] (10.5,2.5)--(10.5,3)--(11,3)--(11,2.5)--(10.5,2.5); 
\draw (10.75,2.75) node {$c$}; 

\draw (7.8,2.8) node {\scriptsize $2$}; 

\draw (10.3,0.8) node {\scriptsize $4$};

\filldraw [fill=black] (8,3) circle [radius=0.1];

\filldraw [fill=black] (10.5,1) circle [radius=0.1];

\end{tikzpicture}
\caption{\footnotesize On this picture we have $p=3,$ $q=2$. On the left we have a box $c\in D$ with $a(c)=5$ and $l(c)=2.$ The corresponding northward arrow belongs to $N_{in}(2),$ and the corresponding westward arrow belongs to $W_{in}(8),$ where $8=2+6\times 2-2\times 3,$ because there are $6=a(c)+1$ westward and $2=l(c)$ northward arrows between the corresponding vertices. Similarly, on right we have $c\in R_{P,Q}\backslash D$ with $a(c)=5$ and $l(c)=3.$ We see that there are $a(c)=5$ westward and $l(c)+1=4$ northward arrows between the corresponding vertices.}
\label{figure:slope->edges}
\end{figure}
\end{proof}

Finally, we use Lemma \ref{lemma:slope->edges} and Theorem \ref{theorem inside rectangle} to deduce Loehr-Warrington's formulas. Let us start with $\ctot_{\frac{q}{p}}:$

$$
c^-_{\frac{q}{p}}(D)=\left| \left\lbrace c\in D:\frac{l(c)+1}{a(c)}=\frac{q}{p}\right\rbrace \right|
=\left| \left\lbrace c\in R_{P,Q}\backslash D:\frac{l(c)+1}{a(c)}=\frac{q}{p}\right\rbrace \right|-
$$
$$ 
\ \ \ \ \ \ \ \ \ \ \ \ \ \ \ \ \ \ \ \ \ \ \ \ \ \ \ \ \ \ \ \ \ \ \ \ \ 
- \left| \left\lbrace (x,y)\in R_{P,Q}\backslash D: qx+py=Kpq-p-q\right\rbrace \right|.
$$
Indeed, points $(x,y)\in R_{P,Q}\backslash D$ such that $qx+py=Kpq-p-q$ are exactly those for which the arrow $(x,y)\rightarrow (P-1,-1)$ has the required slope $\frac{q}{p}$. Note that
$$
\left| \left\lbrace (x,y)\in R_{P,Q}\backslash D: qx+py=Kpq-p-q\right\rbrace \right| = K - |N_{in}(0)|. 
$$
\noindent Indeed, there are exactly $\gcd(P,Q)-1=K-1$ boxes $(x,y)$ in $R_{P,Q},$ such that $qx+py=Kpq-p-q,$ and $|N_{in}(0)|-1$ such boxes inside $D$ (all such boxes correspond to the vertices of the boundary path labeled by $0,$ and we always arrive at such vertices along northward edges). Finally, we subtract one for the initial vertex of the path. We conclude,

$$
\ctot_{\frac{q}{p}}(D)=c^-_{\frac{q}{p}}(D)+c^+_{\frac{q}{p}}(D)
$$
$$
=\left| \left\lbrace c\in R_{P,Q}\backslash D:\frac{l(c)+1}{a(c)}=\frac{q}{p}\right\rbrace \right| - K + |N_{in}(0)| + \left| \left\lbrace c\in D:\frac{l(c)}{a(c)+1}=\frac{q}{p}\right\rbrace \right|
$$
$$
=\sum\limits_{v\in V_M} |W_{in}(v)||N_{in}(v)| - K+|N_{in}(0)|.
$$
\noindent The last equality follows from the Lemma \ref{lemma:slope->edges}.

Statistic $\midd_{\frac{q}{p}}(D)$ can be treated similarly:

$$
\midd_{\frac{q}{p}}(D)=\left| \left\lbrace c\in D:\frac{l(c)}{a(c)+1}<\frac{q}{p}<\frac{l(c)+1}{a(c)}\right\rbrace \right|
$$
$$
=|D|-\left| \left\lbrace c\in D:\frac{l(c)}{a(c)+1}\ge\frac{q}{p}\right\rbrace \right|-\left| \left\lbrace c\in D:\frac{q}{p}\ge\frac{l(c)+1}{a(c)}\right\rbrace \right|
$$
$$
=|D|-\left| \left\lbrace c\in  D:\frac{l(c)}{a(c)+1}\ge\frac{q}{p}\right\rbrace \right|-\left| \left\lbrace c\in R_{P,Q}\backslash D:\frac{q}{p}\ge\frac{l(c)+1}{a(c)}\right\rbrace \right|+\left| R^+_{P,Q}\backslash D \right|
$$
$$
=|R^+_{P,Q}|-\sum\limits_{v,w\in V_M, v\le w} |W_{in}(v)||N_{in}(w)|.
$$
Here we first applied Theorem \ref{theorem inside rectangle southeast}, then Corollary \ref{GM bijection southeast}, and then, finally, Lemma \ref{lemma:slope->edges}. These formulas were first proved by Loehr and Warrington by induction (see Chapter $5$ in \cite{LW}).

\section{Remarks on Geometry.}\label{Section: geometric remarks}

Statistics $\ctot_{\frac{q}{p}}(D),\ c^-_{\frac{q}{p}}(D),\ c^+_{\frac{q}{p}}(D),$ and $\midd_{\frac{q}{p}}(D)$ have nice geometric interpretations in terms of the toric action on the Hilbert schemes of points on the complex plane. The Hilbert scheme $\Hilb ^{n}(\mathbb C^2)$ is the space of ideals of codimension $n$ in the polynomial ring $\mathbb C[x,y].$ It inherits a natural action of the two-dimensional torus $(\mathbb C^*)^2,$ acting by scaling on the variables $x$ and $y.$ The fixed points are the monomial ideals, naturally parametrized by Young diagrams: given a Young diagram $D,$ the corresponding monomial ideal $I_D$ is spanned by the monomials $x^ky^l$ for $(k,l)\in \mathbb (Z_{\ge 0})^2\backslash D.$ The Hilbert polynomials of the tangent spaces at the fixed points were computed by Ellingsrud and Str\o mme in \cite{ES}:

$$
T_{I_D}\mbox{Hilb}^n(\mathbb{C}^2)=\sum_{c\in D}(t_1^{a(c)+1}t_2^{-l(c)}+t_1^{-a(c)}t_2^{l(c)+1}),
$$
where $D$ is a Young diagram and $I_D\subset\mathbb C[x,y]$ is the corresponding monomial ideal. Let $p$ and $q$ be coprime positive integers. Consider the one-dimensional subtorus $T_{p,q}:=\{t^p,t^q\}\subset (\mathbb C^*)^2.$ If $p+q>n,$ then the fixed points of the action of $T_{p,q}$ coincide with the fixed points of the action of the whole torus $(\mathbb C^*)^2$. Indeed, otherwise there should exist a monomial ideal $I_D\in\mbox{Hilb}^n(\mathbb C^2)$ such that at least one of the characters of the torus action on the tangent space is orthogonal to $T_{p,q}.$ In other words, according to the Ellingsrud-Str\o mme's formula, there should exist a Young diagram $D$ with $|D|=n,$ and a box $c\in D,$ such that either $(a(c)+1)p-l(c)q=0$ or $-a(c)p+(l(c)+1)q=0.$ Since $p$ and $q$ are relatively prime and positive, it follows then that $n=|D|\ge a(c)+l(c)+1\ge q+p.$ 

Note also that the orbits of the subgroup stay bounded as $t\rightarrow 0.$ It follows that one can consider the Bia\l ynicki-Birula cell decomposition of the Hilbert scheme by unstable cells (see \cite{BB}). To compute the dimension of the unstable cell $C_D$ one should count how many summands $t_1^kt_2^l$ in the Ellingsrud-Str\o mme's formula correspond to the repelling directions, which is equivalent to the inequality $(k,l)\cdot(p,q)=pk+ql>0.$ Note that from each pair of summands $t_1^{a(c)+1}t_2^{-l(c)}$ and $t_1^{-a(c)}t_2^{l(c)+1}$ at least one always satisfy this inequality. Indeed, 
$$
(a(c)+1,-l(c))\cdot (p,q)+(-a(c),l(c)+1)\cdot (p,q)=(1,1)\cdot (p,q)=p+q>0.
$$
Both summands satisfy the inequality if and only if one has
$$
\frac{a(c)}{l(c)+1}<\frac{q}{p}<\frac{a(c)+1}{l(c)}.
$$
Therefore, the dimension of the unstable cell $C_D$ is given by

$$
\dim C_D = |D|+h_{\frac{q}{p}}(D).
$$

If $p+q\le n,$ then the fixed points of the $T_{p,q}$-action are not isolated. Indeed, in this case it is not hard construct a Young diagram $D$ with $|D|=n,$ and a box $c\in D,$ such that $a(c)=q$ and $l(c)+1=p.$ The fixed point sets are called {\it quasihomogeneous Hilbert schemes} and denoted $\Hilb^n_{p,q}(\mathbb C^2).$ They are smooth and compact, but might be reducible and, moreover, irreducible components might have different dimensions. Similar to the above, one concludes that the dimension of the unstable subvariety of a fixed point $I_D$ is equal to $|D|+\midd_{\frac{q}{p}}(D).$

\begin{lemma}
The dimension of the irreducible component of the $\Hilb^n_{p,q}(\mathbb C^2)$ containing the monomial ideal $I_D$ is equal to $\ctot_{\frac{q}{p}}(D).$
\end{lemma}

\begin{proof}
Indeed, the dimension of the subspace in the tangent space at $I_D$ fixed by the subtorus $T_{p,q}$ is equal to the number of summands $t_1^kt_2^l$ in the Ellingsrud-Str\o mme's formula, such that $(k,l)\cdot(p,q)=0,$ which is exactly $\ctot_{\frac{q}{p}}(D).$
\end{proof}

The factor torus $T^{p,q}:=(\mathbb C^2)/T_{p,q}$ acts on $\Hilb^n_{p,q}(\mathbb C^2)$ with isolated fixed points, which gives rise to two Bia\l ynicki-Birula cell decompositions of $\Hilb^n_{p,q}(\mathbb C^2)$: into stable and into unstable varieties of the fixed points (here one should choose a parametrization $\mathbb C^*\to T^{p,q}$ of the factor torus $T^{p,q}=(\mathbb C^*)^2/T_{p,q};$ we choose $t\mapsto [(t,1)]$). One immediately sees that $c^-_{\frac{q}{p}}(D)$ is the dimension of the stable variety of the fixed point $I_D\in \Hilb^n_{p,q}(\mathbb C^2),$ and $c^+_{\frac{q}{p}}(D)$ is the dimension of the unstable variety. 

Irreducible components of $\Hilb^n_{p,q}(\mathbb C^2)$ were studied by Evain in \cite{E}. 
He showed that two monomial ideals $I_D$ and $I_{D'}$ belong to the same connected component of $\mbox{Hilb}^n_{p,q}(\mathbb C^2)$ if and only if the Young diagrams $D$ and $D'$ have the same weighted content, i.e. for any integer $d$ one has
$$
\left|\{(x,y)\in D:px+qy=d\}\right|=\left|\{(x,y)\in D':px+qy=d\}\right|.
$$

One can reformulate Evain's results to show that two monomial ideals belong to the same irreducible component of $\Hilb^n_{p,q}(\mathbb C^2)$ if and only if the corresponding Young diagrams share the same graph $M(D):$

\begin{lemma}
Two Young diagrams $D$ and $D'$ have the same $(p,q)$-weighted content if and only if $M(D)=M(D').$
\end{lemma}

\begin{proof}
The proof is a manipulation with generating series. Let us use the following notations:
$$
P_N^{M(D)}:=\sum\limits_{v\in V_{M(D)}} |N_{in}(v)|t^v,
$$
$$
P_W^{M(D)}:=\sum\limits_{v\in V_{M(D)}} |W_{out}(v)|t^v,
$$
and
$$
C_{p,q}^D:=\sum\limits_{c\in D} t^{px+qy}.
$$
Note that knowing the polynomials $P_N^{M(D)}$ and $P_W^{M(D)}$ is enough to recover the graph $M(D)$ (in fact, it is enough to know just one of these polynomials, as we will see below). On the other side, knowing $C_{p,q}^D$ is equivalent to knowing the $(p,q)$-weighted content of $D.$ We will deduce explicit formulas for $C_{p,q}^D$ in terms of $P_N^{M(D)},$ and in terms of $P_W^{M(D)},$ which will be evidently invertible. This will be enough to complete the proof. 

Let $e\in N_{in}(v)$ be a northward edge of $M(D).$ It corresponds to a northward edge on the boundary path $B(D).$ It follows that the weighted content of the box immediately to the west from the edge $e$ is equal to $Kpq-p-q-v.$ Therefore, the generating series for the contents of all boxes in the half-row to the west of the edge $e$ equals to
$$
t^{Kpq-p-q-v}(1+t^q+t^{2q}+\ldots)=\frac{t^{Kpq-p-q-v}}{1-t^q}.
$$ 
To get $C_{p,q}^D,$ one should sum up the above formula over all northward edges and subtract the generating series of the weighted contents of all boxes in the half strip $\{(x,y)\in\mathbb Z^2:x<0\ 0\le y<q\},$ which can be computed as follows:
$$
(t^{-q}+t^{-q+p}+\ldots+t^{-q+(Kq-1)p})(1+t^q+t^{2q}+\ldots)=t^{-q}\frac{1-t^{Kpq}}{(1-t^p)(1-t^q)}.
$$
Therefore, one gets
$$
C_{p,q}^D=\sum\limits_{v\in V_{M(D)}} |N_{in}(v)|\frac{t^{Kpq-p-q-v}}{1-t^q}-t^{-q}\frac{1-t^{Kpq}}{(1-t^p)(1-t^q)}
$$ 
$$
=\frac{t^{Kpq-p-q}}{1-t^q}P_N^{M(D)}(t^{-1})-t^{-q}\frac{1-t^{Kpq}}{(1-t^p)(1-t^q)}.
$$
Similarly, one can deduce the following formula in terms of $P_W^{M(D)}:$
$$
C_{p,q}^D=\frac{t^{Kpq-p-q}}{1-t^p}P_W^{M(D)}(t^{-1})-t^{-p}\frac{1-t^{Kpq}}{(1-t^p)(1-t^q)}.
$$
Note that both formulas are invertible.
\end{proof}

The above consideration provides a geometric explanation of the fact that $\ctot_{\frac{q}{p}}(D)$ and $\midd_{\frac{q}{p}}(D)$ depend only on the graph $M(D)$ and not on the Eulerian tour $E(D).$  Moreover, existence of the Loehr-Warrington's bijections interchanging statistics $c^-_{\frac{q}{p}}$ and $c^+_{\frac{q}{p}}$ while preserving the multigraph $M(D),$ follows from the fact that the stable and the unstable cell decompositions of an irreducible component of $\Hilb_{p,q}^n(\mathbb C^2)$ have the same number of cells of a given dimension. However, geometric meaning of a particular bijection constructed in \cite{LW} remains mysterious.

\end{document}